\documentclass{article}
\usepackage{amscd,amsmath,amssymb,amsthm,amsfonts,epsfig,graphics}
\usepackage{cite}
\usepackage{longtable}
\usepackage{tabularx} 
\usepackage{tabu} 
\newtheorem{theorem}{Theorem}
\theoremstyle{plain}

\newtheorem{lemma}{Lemma}

\numberwithin{equation}{section}
\numberwithin{lemma}{section}
\numberwithin{theorem}{section}
\numberwithin{corollory}{section}
\numberwithin{proposition}{section}

\begin{document}

		\title{Unit Groups  of group algebras of certain quasidihedral group}
		\author{
			Suchi Bhatt and Harish Chandra  
			\\	Department of Mathematics and Scientific Computing\\  Madan Mohan Malaviya University of Technology \\Gorakhpur 273010 (U.P.), India.\\
			\textit{E-Mail}   hcmsc@mmmut.ac.in}
		\date{}			
		\maketitle
		
		\begin{abstract}
				Let $F_{q}$ be any finite field of characteristic $p>0$ having $q = p^{n}$ elements. In this paper, we have obtained the complete structure of unit groups of group algebras $F_{q}[QD_{2^k}]$,  for   $k = 4$ and $5$, for any prime  $p>0$, where $QD_{2^k}$ is quasidihedral group of order $2^k$.
			\end{abstract}

		\smallskip
		\noindent  \textit{Mathematics Subject Classification \text{(2010)}}: 16S34, 17B30.
		
		\smallskip
		\noindent \text{Key words:}\; Group algebras, Unit groups,  Jacobson radical.

      \section{Introduction}
    Let $FG$ be the group algebra of a group $G$ over a field $F$. Let $N$ be the normal subgroup of $G$.  The  natural homomorphism $g\mapsto gN$ such that $G \mapsto G/N$ can be extended to an algebra homomorphism from $FG \mapsto F[G/N]$ defined by 
    $$ \sum_{g\in G} a_{g}g \rightarrow \sum_{g\in G} a_{g}gN  $$ for $a_{g}\in F$.
    The kernal of this $F$-algebra homomorphism is $\omega(N)$, which is an ideal generated by $\{ x-1, \, x\in N\}$ in $FG$. 
    The augmentation ideal $\omega(FG)$ of the group algebra $FG$ is defined as:
    $$ \omega(FG) = \bigg \{  \sum_{g\in G} a_{g}g \in FG \, |  \,  a_{g} \in F, \sum_{g\in G} a_{g} = 0 \bigg\}  $$
    Obviously, 
    $  \frac{FG}{\omega(N)} \cong  F(\frac{G}{N})$.
    It is observed  that $\omega(N) =\omega(FN)FG = FG\omega(FN)$. Now $\frac{FG}{\omega(G)} \cong F$ implies that $J(FG)\subseteq \omega(FG)$, where $J(FG)$ is Jacobson radical of $FG$. It is well known that for an ideal $I\subseteq J(FG)$, the natural homomorphism $FG$ to ${FG}/{I}$ induces an epimorphism from the unit group of $FG$, $U(FG)$ to $U({FG}/{I})$  with kernal  $1+I$ and 
    $$\frac{U(FG)}{1+I} \cong   U \left ( \frac{FG}{I}\right )$$                                
    We use $V_{1}$ as kernal of the epimorphism $U(FG)$ to $U({FG}/{I})$ and $V_{1} =  1+ J(FG)$.
    For any group $G$ with $g_{1}, g_{2} \in G$ the commutator $(g_{1},g_{2}) = g^{-1}_{1}g^{-1}_{2}g_{1}g_{2}$.
    The lower central series of $G$ is defined as $$G = \gamma_{1}(G) \supseteq \gamma_{2}(G) \supseteq... \supseteq \gamma_{m}(G) \supseteq...$$ where $\gamma_{c+1}(G) = (\gamma_{c}(G), G)$ is the group generated by $(g,h)$,  $g\in \gamma_{c}(G), \, h\in G$, for $c\geq 1$. A group $G$ is said to be nilpotent of class $c$ if $\gamma_{c+1}(G) = 1$ but $\gamma_{c}(G)\neq 1$.\\
    
    Let $F$ be a finite field of characteristic $p>0$.  An element $g\in G$ is called $p$-regular if $(p,o(g)) = 1$, where Char$F$ = $p>0$.
    Suppose $m$ be the L.C.M. of the orders of  $p$-regular elements of G and $\eta$ be a primitive $m$-th root of unity.  Now if $T$ be the multiplicative group consisting of those integers $t$, taken modulo $m$ such that  $\eta  \rightarrow \eta^{t}$ is an  F-automorphism of $F(\eta)$ over F. Any  two $p$-regular elements $g_1, g_2 \in G$ are said to be $F$-conjugate if $g_1^{t} = x^{-1}g_2x$, for some $x\in G$ and $t\in T$. 
    It gives an equivalence relation which partitions the $p$-regular elements of $G$  into $p$-regular, $F$-conjugacy classes. In accordance of  Witt-Berman theorem \cite[Ch.17, Theorem 5.3]{gkarp1}, we have the number of non-isomorphic simple $FG$-modules  equal to the number of $F$-conjugacy classes of $p$-regular elements of G.
    
    Problem based on the structure of unit group $U(FG)$  has  generated considerable interest. A lot of papers has been appeared in this direction  (see \cite{Lc,jg1,joeg,kaukhan,khanss2,nmss01,mss1,mss2,PPS1,Cpm1}).  Sharma and  Srivastava  \cite{ssm3,ssm4,khanss1}  have obtained the  structure of the unit group of $FG$ for $G = S_3, \, S_4\, $ and $A_4$.
    
    In this paper we have obtained  the structure of unit groups of group algebras of quasi dihedral groups $QD_{16}$ and  $QD_{32}$ of order 16 and 32 respectively. The  presentation of quasidihedral groups are given as
    $$ QD_{2^k} = < a, x\,    |   \, a^{2^{k-1}} = x^2 = 1, \,  xax = a^{2^{k-2}-1}  >.$$
   
    The  distinct conjugacy classes of $QD_{16}$ are  ${C}_{1} = \{1\}$, ${C}_{2} = \{a^{4}\}$, ${C}_{3} = \{a^{2}, a^{6}\}$, ${C}_{4} = \{a, a^{3}\}$, ${C}_{5} = \{a^{5}, a^{7}\}$, ${C}_{6} = \{x, a^{2}x, a^{4}x, a^{6}x\}$, ${C}_{7} = \{ax, a^{3}x, a^{5}x, a^{7}x\}$. Also $\widehat{{C}_{i}} $, denotes the class sum of ${C}_{i}$, where $ 1\leq i \leq 7$.\\
    
     The  distinct conjugacy classes of $QD_{32}$ are  ${C'}_{1} = \{1\}$, ${C'}_{2} = \{a^{8}\}$, ${C'}_{3} = \{a, a^{7}\}$, $r{C'}_{4} = \{a^{\pm2}\}$, ${C'}_{5} = \{a^{3}, a^{5}\}$, ${C'}_{6} = \{a^{\pm4}\}$, ${C'}_{7} = \{a^{\pm6}\}$,  ${C'}_{8} = \{a^{-1}, a^{-7}\}$, ${C'}_{9} = \{a^{-3}, a^{-5}\}$, \sloppy ${C'}_{10} = \{x,a^{2}x,a^{4}x,a^{6}x,a^{8}x,a^{10}x,a^{12}x,a^{14}x\}$, ${C'}_{11} = \{ax,a^{3}x,a^{5}x,a^{7}x,a^{9}x,a^{11}x,a^{13}x,a^{15}x\}$. Also $\widehat{{C'}_{j}} $, denotes the class sum of ${C'}_{j}$, where $ 1\leq j \leq 11$.\\
     
    Throughout the paper,  $M(n, F)$ denotes the algebra of all $n\times n$ matrices over $F$, $F^{*}= F \-\setminus \{0\}$,
    $GL(n,F)$ is the general linear group of degree $n$ over $F$, $CharF$ is the characteristic of $F$ and $C_{n}$ is the cyclic group of order $n$.
    \section{Preliminaries}
   % \section{Unit Group of $F_{q}[QD_{16}]$}
    In this section, we give a complete characterization of unit group $U(F_{q}[QD_{2^k}])$, for $k= 4$ and $5$ having char$F_{q}$ =  $p>0$.
    
    We shall use the following results frequently  in our work.
    
    \begin{lemma} \cite{Cpm2} \label{l1}
    	Let $G$ be a group and  $R$ be a commutative ring. Then the set of all finite class sums forms an $R$-basis of $ \zeta(RG)$, the center
    	of $RG$.
    	
    \end{lemma}
    
    \begin{lemma} \cite{Cpm2} \label{l2}
    	Let $FG$ be a semi-simple group algebra. If $G'$ denotes the
    	commutator subgroup of $G$, then
    	$$  FG = FG_{e_{G'}} \oplus \Delta(G,G')$$
    	where $FG_{e_{G'}} \cong F(G/G')$ is the sum of all commutative simple components of $FG$ and $\Delta(G,G')$ is the sum of all the others. 
    	
    \end{lemma}
    \begin{lemma}\cite[Theorem 7.2.7]{PPS2} \label{l3}
    	Let $H$ be a normal subgroup of $G$ with $[G : H] = n < \infty$. Then $(J(FG))^n \subseteq J(FH)FG \subseteq J(FG)$. If in addition, $n\neq 0$ in
    	$F$, then  $ J(FG) = J(FH)FG$.
    \end{lemma}
    
    \begin{lemma} \cite[Lemaa 1.17]{gka2} \label{l4}
    	Let G be a locally finite $p$-group and let $F$ be a  field of characteristic $p$. Then $J(FG) = \omega(FG)$.
    \end{lemma}

     \section{Unit Group of $F_{q}[QD_{16}]$}

    \begin{theorem} \label{t1}
    	Let $U(F_{q}[QD_{16}])$ be the unit group of group algebra $F_{q}[QD_{16}]$ of quasidihedral group of order $16$ over any finite field of positive characteristic $p$. Let $V = 1+J(F_{q}[QD_{16}])$, where $J(F_{q}[QD_{16}])$ denotes the Jacobson radical of the group algebra $F_{q}[QD_{16}]$.
    	\begin{enumerate}
    		\item If $|F_q| = 2^n$, then \\
    		(a) $\frac{U(F_{q}[QD_{16}])}{V_{1}}$ $\cong C_{q-1}$.\\
    		
    		(b) $V$ is 2-group of order $2^{15n}$ and exponent 8.\\
    		
    		(c) Nilpotency  class of $V$ is 4.\\
    		
    		(d) $V$ is centrally metabelian.  \\
    		
    		\item If $|F_{q}| = p^n$ where $p>2$, then $U(F_{q}[QD_{16}])$ is isomorphic to:\\
    		
    		(a) $C^{4}_{q-1} \times  GL(2,F_{q})^{3}$,  if $q\equiv  1$ or $3 \, mod\,8$.\\
    		
    		(b) $C^{4}_{q-1} \times  GL(2,F_{q})\times GL(2,F_{q^2})$,                          if $q\equiv  -1$ or $-3 \,\,mod\,8$.
    	\end{enumerate}
    	
    \end{theorem}
    \begin{proof}
    	\textbf{1.(a)} Let Char$F_{q}$ = 2 and $|F_{q}| = 2^n$. Since $D_{8}$ is normal subgroup of $QD_{16}$ of index 2, which is not invertible. We have  $\frac{QD_{16}}{D_{8}} \cong C_{2}$,  
    	$F_{q}C_{2} \cong F_{q}[\frac{QD_{16}}{D_{8}}] \cong \frac{F_{q}[QD_{16}]}{\omega{(D_{8})}}$, so that $dim_{F_{q}}(\omega(D_{8})) = 14$. Further $\omega(D_{8})^5 = 0$ and as $QD_{16}$ is a 2-group, therefore by Lemma \ref{l4}, $\omega(F_{q}[QD_{16}]) = J(F_{q}[QD_{16}])$. Hence 
    	$$ J(F_{q}C_{2}) \cong J \left( \frac{F_{q}[QD_{16}]}{\omega(D_{8})} \right) \cong \frac{J(F_{q}[QD_{16}])}{\omega(D_{8})}.$$
    	
    	It is known that $dim_{F_{q}}(J(F_{q}C_{2})) = 1$ and $J(F_{q}C_{2})^2 = 0$, which implies that $dim_{F_{q}}(J(F_{q}[QD_{16}])) = 15$ and $ J(\frac{F_{q}[QD_{16}]}{\omega(D_{8})})^2 = 0$. Hence we have $dim_{F_{q}}(\frac{F_{q}[QD_{16}]}{J(F_{q}[QD_{16}])}) = 1$. Since $\frac{F_{q}[QD_{16}]}{J(F_{q}[QD_{16}])}$ is commutative, therefore $$\frac{U(F_{q}[QD_{16}])}{V} \cong F^*_{q}\cong C_{q-1}.$$
    	Now the complete description for $V$ is as follows:
    	
    	\textbf{(b)} It can be easily  seen that $y^8 =1$, for all $y\in V$, hence exponent of $V$ is 8.
    	
    	\textbf{(c)} Since $V = 1+J(F_{q}[QD_{16}])$ and $dim(J(F_{q}[QD_{16}])) = 15$ therefore $|V| = 2^{15n}$ and $V$ is a 2-group.  Now let $a, \, b \in J(F_{q}[QD_{16}])$. 
    	If we take $u_{1} \equiv (1+a, 1+b)$, then $u_{1} \equiv 1+ba((a,b)-1)$  mod $(J^3)$,
    	$u_{2}= (u_{1}, 1+b)\equiv 1+b^2a(a,b)((u_1,b) -1)$ mod $(J^5)$,
    	$u_{3}=(u_{2}, 1+b)\equiv 1+b^3a(a,b)((u_1,b)-1)((u_{2},b)-1)$  mod $(J^7)$ and $u_{4} \equiv 1 $, therefore $\gamma_{5}(V_1) = 1$  (cf. see \cite{gkarp2}). Hence $V$ is nilpotent group of class 4.\\
    	
    	\textbf{(d)} Now $\frac{U(F_{q}[QD_{16}])}{V}$ is an abelian group, thus $U(F_{q}[QD_{16}])'\subseteq V$. Also $U(F_{q}[QD_{16}])'' \subseteq V^{'} \subseteq \zeta(F_{q}[QD_{16}])$, therefore $U(F_{q}[QD_{16}])$ is centrally metabelian. \\
    	
    	\textbf{2.} Let Char$F_{q} = p \,(>2)$. Now since $p$ does not divides $|QD_{16}|$, therefore by Maschke's theorem $F_{q}[QD_{16}]$ is semi-simple over $F_{q}$ and thus $J(F_{q}[QD_{16}]) = 0$. Hence $ \frac{F_{q}[QD_{16}]}{J(F_{q}[QD_{16}])} \cong F_{q}[QD_{16}]$. Now using Wedderburn structure theorem, we have:
    	$$ F_{q}[QD_{16}] \cong  (\bigoplus^{c}_{i=1}  M(n_{i},K_{i}))$$
    	where $K_{i}$' s are finite dimensional division algebras over $F_{q}$ and hence $K_{i}$' s are  finite extensions of $F_{q}$.\\ 
    	Since $p>2$, therefore $F_{q}\big(\frac{QD_{16}}{QD'_{16}}\big) \cong F_{q}(C_{2}\times C_{2}) \cong F_{q}\oplus F_{q}\oplus F_{q}\oplus F_{q}$. Thus by  Wedderburn structure theorem and by Lemma \ref{l2}, we have $$ F_{q}[QD_{16}] \cong F_{q}\oplus F_{q}\oplus F_{q}\oplus F_{q} \oplus  (\bigoplus^{r}_{i=1}  M(n_{i},K_{i})).$$
    	Hence $$\zeta(F_{q}[QD_{16}])\cong F_{q}\oplus F_{q}\oplus F_{q}\oplus F_{q} \oplus \big(\bigoplus^{r}_{i=1}  K_{i})$$
    	and by Lemma \ref{l1}, $dim_{F_{q}}(\zeta(F_{q}[QD_{16}])) = 7$, therefore  $\sum_{i=1}^{r} [K_{i} : F_{q}] = 7-4 = 3$.
    	
    	Now  for any $s\in N$, $x^{q^s} = x$,  $\forall x\in \zeta(F_{q}[QD_{16}])$ if and only if $\widehat{{C}^{q^s}_{i}} = \widehat{{C}_{i}}$, for all $1\leq i\leq 7$. This holds good if and only if $8|q^s-1$ or $8|q^s+1$. Suppose $K_{i}^* = <y_{i}>$ for all $i$, $1\leq i\leq r$, then $x^{q^s} = x$,  $\forall x\in \zeta(F_{q}[QD_{16}])$ if and only if $y^{q^{s}}_{i} = 1$, which is equivalent to $[K_{i}:F_{q}]|s$, for all $1\leq i\leq r$. Hence the least number $l$ such that $8|q^l-1$ or $8|q^l+1$ can be given as
    	$$l = l.c.m.\{[K_{i}:F_{q}] |  1\leq i \leq r \}.$$
    	By  calculation, we have the following possibilities for $q$
    	\begin{enumerate}
    		\item  If $q \equiv 1 \, mod\, 8$, then $l = 1$,
    		\item If $q \equiv -1 \, mod\, 8$, then $l = 2$,
    		\item If $q \equiv 3 \, mod\, 8$, then $l = 1$,
    		\item If $q \equiv -3 \, mod\, 8$, then $l = 2$.
    	\end{enumerate}
    	
    	To find the number of simple components in the Wedderburn  decomposition of $F_{q}[QD_{16}]$, we apply the Witt-Berman theorem. Here $m =8$. Let $c\,(=r+4)$ is the number of simple components. First we will find $T$ and $p$-regular $F_{q}$ conjugacy classes as described in introduction. \\
    	\textbf{(a)} $\mathbf{q \equiv 1 \, mod \, 8}$\\
    	$T \equiv \,\{1\} \, mod \, 8$. Thus $ C_{i}$, $1 \leq i \leq 7$ will be $p$ regular $F_{q}$-conjugacy classes and hence $c= 7$.\\
    	\textbf{(b)} $\mathbf{q \equiv -1 \, mod \, 8}$\\
    	$ T \equiv \,\{1,7\} \, mod \, 8$. Thus  $\{1\}$, $\{a^{4}\}$, $\{a,a^{3},a^{5},a^{7} \}$, $\{a^{2},a^{6} \}$, \sloppy $\{x,a^{2}x,a^{4}x,a^{6}x \}$, $\{ax,a^{3}x,a^{5}x,a^{7}x \}$ will be $p$ regular $F_{q}$-conjugacy classes and hence $c= 6$.\\
    	\textbf{(c)} $\mathbf{q \equiv 3 \, mod \, 8}$\\
    	$ T \equiv \,\{1, 3\} \, mod \, 8$. Thus $\{1\}$, $\{a^{4}\}$, $\{a,a^{3} \}$, $\{a^{2},a^{6} \}$, $\{a^{5},a^{7} \}$, $\{x,a^{2}x,a^{4}x,a^{6}x \}$, $\{ax,a^{3}x,a^{5}x,a^{7}x \}$ will be $p$ regular $F_{q}$-conjugacy classes and hence $c= 7$.\\
    	\textbf{(d)}$\mathbf{q \equiv -3 \, mod \, 8}$\\
    	$ T \equiv \,\{1, 5\} \, mod \, 8$. Thus $\{1\}$, $\{a^{4}\}$, $\{a,a^{3},a^{5},a^{7} \}$, $\{a^{2},a^{6} \}$, $\{x,a^{2}x,a^{4}x,a^{6}x \}$, $\{ax,a^{3}x,a^{5}x,a^{7}x \}$ will be $p$ regular $F_{q}$-conjugacy classes and hence $c= 6$.\\
    	Thus from the above cases, we have following possibilities for $S = ([K_{i}:F_{q}])^{r}_{i = 1}$, depending  on $q$:
    	\begin{enumerate}
    		\item  $q \equiv 1 \, mod \, 8 \implies S \, = (1,1,1)$,
    		\item  $q \equiv -1 \, mod \, 8 \implies S \, = (1,2)$,
    		\item  $q \equiv 3 \, mod \, 8 \implies S \, = (1,1,1)$,
    		\item  $q \equiv -3 \, mod \, 8 \implies S \, = (1,2)$.
    	\end{enumerate}
    	Due to  dimensions constraint $n_{i} = 2$, $\forall 1\leq i\leq r$. Therefore
    	\[
    	F_{q}[QD_{16}] \cong \left.
    	\begin{cases}
    	F_{q} \oplus F_{q}\oplus  F_{q}\oplus F_{q}\oplus  M(2,F_{q})^{3},  & \text{if},    q\equiv \, 1\, \text{or} \,3\, mod\,8\\
    	
    	F_{q} \oplus F_{q}\oplus F_{q}\oplus F_{q} \oplus  M(2,F_{q})\oplus M(2,F_{q^{2}}),                                                   & \text{if}, q\equiv \, -1 \, \text{or} \, -3\, mod\,8

    	\end{cases} \right. \]
    
    	Hence the result follows.
    	\end{proof}
    	  \section{Unit Group of $F_{q}[QD_{32}]$}
    	  \begin{theorem}
    	  	Let $U(F_{q}[QD_{32}])$ be the unit group of group algebra $F_{q}[QD_{32}]$ of quasidihedral group of order $32$ over any finite field of positive characteristic $p$. Let $V_{1} = 1+J(F_{q}[QD_{32}])$, where $J(F_{q}[QD_{32}])$ denotes the Jacobson radical of the group algebra $F_{q}[QD_{32}]$.
    	  	\begin{enumerate}
    	  		\item If $|F_q| = 2^n$, then \\
    	  		(a) $\frac{U(F_{q}[QD_{32}])}{V_{1}}$ $\cong C_{q-1}$.\\
    	  		
    	  		(b) $V_{1}$ is 2-group of order $2^{31n}$ and exponent 16.\\

    	  		\item If $|F_{q}| = p^n$ where $p>2$, then $U(F_{q}[QD_{32}])$ is isomorphic to:\\
    	  		
    	  		 (a) $C^{4}_{q-1} \times  GL(2,F_{q})^{7}$,  if $q\equiv  1$ or $ 7 \, mod\,16$.\\
    	  		
    	  		(b) $C^{4}_{q-1} \times  GL(2,F_{q})^{3}\times GL(2,F_{q^2})^{2}$,                          if $q\equiv  -1$ or $-7 \,\,mod\,16$.\\
    	  		
    	  		 (c) $C^{4}_{q-1} \times  GL(2,F_{q})\times GL(2,F_{q^2}) \times GL(2,F_{q^4})$,                          if $q\equiv  \pm3$ or $\pm5 \,\,mod\,16$.\\
    	  		 
    	  		% (d) $C^{4}_{q-1} \times  GL(2,F_{q})\times GL(2,F_{q^2}) \times GL(2,F_{q^4})$,                          if $q\equiv  -3$ or $5 \,\,mod\,16$.
    	  	\end{enumerate}
    	  	
    	  \end{theorem}
      
    	  \begin{proof}
    	  \textbf{1(a)} Let Char$F_{q}$ = 2, then $|F_{q}| = 2^n$. Since $C_{16}$ is normal subgroup of $QD_{32}$, which is not invertible. Therefore  $\frac{QD_{32}}{C_{16}} \cong C_{2}$,  
    	  $F_{q}C_{2} \cong F_{q}[\frac{QD_{32}}{C_{16}}] \cong \frac{F_{q}[QD_{32}]}{\omega{(C_{16})}}$ and then  $dim_{F_{q}}(\omega(C_{16})) = 30$. Now as $QD_{32}$ is a 2-group, therefore by Lemma \ref{l4}, $\omega(F_{q}[QD_{32}]) = J(F_{q}[QD_{32}])$. Hence 
    	  $$ J(F_{q}C_{2}) \cong J \left(\frac{F_{q}[QD_{32}]}{\omega(C_{16})}\right) \cong \frac{J(F_{q}[QD_{32}])}{\omega(C_{16})}.$$
    	  	It is known that $dim_{F_{q}}(J(F_{q}C_{2})) = 1$ and $J(F_{q}C_{2})^2 = 0$, which implies that $dim_{F_{q}}(J(F_{q}[QD_{32}])) = 31$ and $ J(\frac{F_{q}[QD_{32}]}{\omega(C_{16})})^2 = 0$. Hence we have $dim_{F_{q}}(\frac{F_{q}[QD_{32}]}{J(F_{q}[QD_{32}])}) = 1$. Since $\frac{F_{q}[QD_{32}]}{J(F_{q}[QD_{32}])}$ is abelian, therefore $$\frac{U(F_{q}[QD_{32}])}{V_1} \cong F^*_{q}\cong C_{q-1}.$$
    	  The complete description for $V_1$ is as follows:\\
    	 \textbf{1(b)} Since $dim_{F_{q}}(J(F_{q}[QD_{32}])) = 31$, therefore $V_1$ is a 2-group of order $2^{31n}.$ Also for every element $y \in V_{1}$, we have $y^{16} =1$. Thus exponent of $V_1$ is 16. Now as $\frac{U(F_{q}[QD_{32}])}{V_1}$ is an abelian group, thus $U(F_{q}[QD_{32}])'\subseteq{V_1}$ and hence $U(F_{q}[QD_{32}])''\subseteq{V'_1}$.\\
    	  
    	   \textbf{2.} Let Char$F_{q} = p \,(>2)$. Now as $p$ does not divides $|QD_{32}|$, thus by Maschke's theorem $F_{q}[QD_{32}]$ is semi-simple over $F_{q}$ and then $J(F_{q}[QD_{32}]) = 0$. Hence $ \frac{F_{q}[QD_{32}]}{J(F_{q}[QD_{32}])} \cong F_{q}[QD_{32}]$. By using Wedderburn structure theorem, we have:
    	  $$ F_{q}[QD_{32}] \cong  (\bigoplus^{c}_{i=1}  M(n_{i},K_{i}))$$
    	  where $K_{i}$' s are finite dimensional division algebras over $F_{q}$ and hence $K_{i}$' s are  finite extensions of $F_{q}$. \\
    	  Since $p>2$, therefore $F_{q}\big(\frac{QD_{32}}{QD'_{32}}\big) \cong F_{q}(C_{2}\times C_{2}) \cong F_{q}\oplus F_{q}\oplus F_{q}\oplus F_{q}$. Therefore from  Wedderburn structure theorem and by Lemma \ref{l2}, we have $$ F_{q}[QD_{32}] \cong F_{q}\oplus F_{q}\oplus F_{q}\oplus F_{q} \oplus  (\bigoplus^{r}_{i=1}  M(n_{i},K_{i})).$$
    	  Hence $$\zeta(F_{q}[QD_{32}])\cong F_{q}\oplus F_{q}\oplus F_{q}\oplus F_{q} \oplus \big(\bigoplus^{r}_{i=1}  K_{i})$$
    	  and by Lemma \ref{l1}, $dim_{F_{q}}(\zeta(F_{q}[QD_{16}])) = 11$, therefore  $\sum_{i=1}^{r} [K_{i} : F_{q}] = 11-4 = 7$.
    	  
    	  Now with similar arguments  as in Theorem \ref{t1},  we have  the least number $l$ such that $16|q^l-1$ or $16|q^l+1$,  can be given as
    	  $$l = l.c.m.\{[K_{i}:F_{q}] |  1\leq i \leq r \}.$$
    	  By  calculation, we have the following possibilities for $q$
    	  \begin{enumerate}
    	  	\item  If $q \equiv 1$ or  $7\, mod\, 16$, then $l = 1$,
    	  	\item If $q \equiv -1$ or $-7 \,mod\, 16$, then $l = 2$,
    	  	\item If $q \equiv \pm3$ or  $\pm 5\, mod\, 16$, then $l = 4$.
    	  	
    	  \end{enumerate}
      Now we will find  out the number of simple components in the Wedderburn decomposition of $F_{q}[QD_{32}]$. We  use the Witt-Berman theorem. Here $m = 16$. If $c\,  (= r+4)$ be the number of simple components, then we will find $T$ and $p$-regular $F_{q}$ conjugacy classes as described in introduction.\\
      \textbf{(a)} $\mathbf{q \equiv 1 \, mod \, 16}$\\
      $T \equiv \,\{1\} \, mod \, 16$. Thus $ C'_{j}$, $1 \leq j \leq 11$ will be $p$ regular $F_{q}$-conjugacy classes and hence $c= 11$.\\
      \textbf{(b)} $\mathbf{q \equiv -1 \, mod \, 16}$\\
      $ T \equiv \,\{-1,1\} \, mod \, 16$. Thus  $\{1\}$, $\{a^{\pm1}, a^{\pm7}\}$, $\{a^{\pm2}\}$, $\{a^{\pm3},a^{\pm5} \}$,   $\{a^{\pm4}\}$, $\{a^{\pm6}\}$ $\{a^{8}\}$,    $\{x,a^{2}x,a^{4}x,a^{6}x,a^{8}x,a^{10}x,a^{12}x,a^{14}x  \}$, $\{ax,a^{3}x,a^{5}x,a^{7}x,a^{9}x,a^{11}x,a^{13}x,a^{15}x \}$  will be $p$ regular $F_{q}$-conjugacy classes and hence $c= 9$.\\
      	\textbf{(c)} $\mathbf{q \equiv 3 \, or\, -5\, mod \, 16}$\\
      $ T \equiv \,\{1, 3, 9, 11\} \, mod \, 16$. Thus $\{1\}$, $\{a^{\pm1}, a^{\pm3}, a^{\pm5}, a^{\pm7}\}$, $\{a^{\pm2},a^{\pm6}\}$, $\{a^{\pm4} \}$,  \sloppy  $\{a^{8}\}$, $\{x,a^{2}x,a^{4}x,a^{6}x,a^{8}x,a^{10}x,a^{12}x,a^{14}x  \}$, $\{ax,a^{3}x,a^{5}x,a^{7}x,a^{9}x,a^{11}x,a^{13}x,a^{15}x \}$  will be $p$ regular $F_{q}$-conjugacy classes and hence $c= 7$.\\
      \textbf{(d)} $\mathbf{q \equiv -3 \, or\, 5\, mod \, 16}$\\
      $ T \equiv \,\{1, 5, 9, 13\} \, mod \, 16$. Thus $\{1\}$, $\{a^{\pm1}, a^{\pm3}, a^{\pm5}, a^{\pm7}\}$, $\{a^{\pm2},a^{\pm6}\}$, $\{a^{\pm4} \}$,  \sloppy  $\{a^{8}\}$, $\{x,a^{2}x,a^{4}x,a^{6}x,a^{8}x,a^{10}x,a^{12}x,a^{14}x  \}$, $\{ax,a^{3}x,a^{5}x,a^{7}x,a^{9}x,a^{11}x,a^{13}x,a^{15}x \}$  will be $p$ regular $F_{q}$-conjugacy classes and hence $c= 7$.\\
      \textbf{(e)} $\mathbf{q \equiv 7 \, mod \, 16}$\\
     $T \equiv \,\{1,7\} \, mod \, 16$. Thus $ C'_{j}$, $1 \leq j \leq 11$ will be $p$ regular $F_{q}$-conjugacy classes and hence $c= 11$.\\
     \textbf{(f)} $\mathbf{q \equiv -7 \, mod \, 16}$\\
     $ T \equiv \,\{1,9\} \, mod \, 16$. Thus  $\{1\}$, $\{a^{\pm1}, a^{\pm7}\}$, $\{a^{\pm2} \}$, $\{a^{\pm4}\}$,  $\{a^{\pm3},a^{\pm5} \}$, $\{a^{\pm6}\}$, $\{a^{8}\}$ \sloppy  $\{x,a^{2}x,a^{4}x,a^{6}x,a^{8}x,a^{10}x,a^{12}x,a^{14}x  \}$, $\{ax,a^{3}x,a^{5}x,a^{7}x,a^{9}x,a^{11}x,a^{13}x,a^{15}x \}$ will be $p$ regular $F_{q}$-conjugacy classes and hence $c= 9$.\\
     	Thus from the above cases, we have following possibilities for $S = ([K_{i}:F_{q}])^{r}_{i = 1}$, depending  on $q$:
     \begin{enumerate}
     	\item  $q \equiv 1$ or $7\, mod \, 16 \implies S \, = (1,1,1,1,1,1,1)$,
     	\item  $q \equiv -1$ or $-7\, mod \, 16 \implies S \, = (1,1,1,2,2)$,
     	\item $q \equiv \pm3$ or $\pm5\, mod \, 16 \implies S \, = (1,2,4)$.

     \end{enumerate}
 
 	Due to  dimensions constraint $n_{i} = 2$, $\forall 1\leq i\leq r$. Therefore
 \[
 F_{q}[QD_{32}] \cong \left.
 \begin{cases}
 F_{q} \oplus F_{q}\oplus  F_{q}\oplus F_{q}\oplus  M(2,F_{q})^{7},  & \text{if},    q\equiv \, 1\, \text{or} \,7\, mod\,16\\
 
 F_{q} \oplus F_{q}\oplus F_{q}\oplus F_{q} \oplus  M(2,F_{q})^3\oplus M(2,F_{q^{2}})^2,                                                   & \text{if}, q\equiv \, -1 \, \text{or} \, -7\, mod\,16\\
 
  F_{q} \oplus F_{q}\oplus F_{q}\oplus F_{q} \oplus  M(2,F_{q})\oplus M(2,F_{q^{2}}) \oplus M(2,F_{q^{4}}),                                                   & \text{if}, q\equiv \, \pm 3 \, \text{or} \, \pm5\, mod\,16

 \end{cases} \right. \]

Hence the result follows.
    \end{proof}
     
\end{document}